\documentclass[runningheads]{llncs}

\usepackage{graphicx}
\newif\ifsubmit

\usepackage[T1]{fontenc}
\usepackage{amsmath}
\usepackage{amssymb}
\usepackage{amsfonts}
\usepackage[english]{babel}
\usepackage{graphicx}
\usepackage{subcaption} 
\usepackage[compatibility=false]{caption} 
\usepackage{bbold}
\usepackage{wrapfig}

\usepackage{diagbox}

\usepackage{wrapfig}
\usepackage[hidelinks]{hyperref}

\usepackage{floatrow}
\newfloatcommand{capbtabbox}{table}[][\FBwidth]

\usepackage{cleveref}

\crefname{equation}{Eq.}{Eq.}
\crefname{pluralequation}{Eqs.}{Eqs.}

\crefname{figure}{Fig.}{Fig.}
\crefname{pluralfigure}{Figs.}{Figs.}

\crefname{section}{Sect.}{Sect.}
\crefname{pluralsection}{Sects.}{Sects.}

\crefname{appendix}{App.}{App.}
\crefname{pluralappendix}{Apps.}{Apps.}

\crefname{table}{Tab.}{Tab.}
\crefname{pluraltable}{Tabs.}{Tabs.}

\crefname{definition}{Def.}{Def.}
\crefname{pluraldefinition}{Defs.}{Defs.}

\crefname{theorem}{Theorem}{Theorems}
\crefname{pluraltheorem}{Theorems}{Theorems}

\crefname{lemma}{Lemma}{Lemma}
\crefname{plurallemma}{Lemmas}{Lemmas}

\crefname{example}{Example}{Example}
\crefname{pluralexample}{Examples}{Examples}

\crefname{assumption}{Assumption}{Assumption}
\crefname{pluralassumption}{Assumptions}{Assumptions}

\crefname{remark}{Remark}{Remark}
\crefname{pluralremark}{Remarks}{Remarks}

\usepackage{tikz}
\usepackage{pgfplots}
\usepackage{pgfplotstable}
\pgfplotsset{compat=1.17}
\usetikzlibrary{shapes.gates.logic.US,shapes.geometric,positioning,fit,calc}
\usepgfplotslibrary{colorbrewer} 

\makeatletter
\def\orcidID#1{\smash{\href{https://orcid.org/#1}{\protect\raisebox{-1.25pt}{\protect\includegraphics{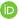}}}}}
\makeatother
\usepackage{ulem} 

\usepackage{stmaryrd}  
\usepackage{mathrsfs}  

\newcommand{\myparagraph}[1]{\medskip\noindent{\bf #1}}

\newcommand{\colorpar}[3]{\colorbox{#1}{\parbox{#2}{#3}}}
\newcommand{\marginremark}[3]{\marginpar{\colorpar{#2}{7em}{\color{#1}#3}}}

\ifsubmit
  \newcommand{\mv}[1]{}
  \newcommand{\ms}[1]{}
  \newcommand{\ml}[1]{}
  \newcommand{\mh}[1]{}
  \newcommand{\bp}[1]{}
\else
  \newcommand{\mv}[1]{\marginremark{black}{yellow}{\tiny{[MV]~ #1}}}
  \newcommand{\ms}[1]{\marginremark{purple}{white}{\tiny{[MS]~ #1}}}
  
  \newcommand{\mh}[1]{\marginremark{white}{blue}{\tiny{[MH]~#1}}}
  \newcommand{\bp}[1]{\marginremark{blue}{white}{\tiny{[BP]~#1}}}
\fi

\newcommand{\RR}{\mathbb{R}}

\renewcommand{\emptyset}{\varnothing}

\newcommand{\fuzznum}{\mathbb{F}} 
\newcommand{\basicevents}{\mathit{BE}}
\renewcommand{\emph}[1]{{\it #1}}

\begin{document}
\title{Fuzzy Fault Trees: the Fast and the Formal}

%
%
\author{
Thi Kim Nhung Dang\inst{1}\orcidID{0000-0002-3235-5952} \and Benedikt Peterseim\inst{2}\orcidID{0009-0004-3510-4325} \and
Milan Lopuhaä-Zwakenberg\inst{2}\orcidID{0000-0001-5687-854X} \and
Mariëlle Stoelinga\inst{2,3}\orcidID{0000-0001-6793-8165}}

\authorrunning{Dang et al.}
%
\institute{
Independent scholar, the Netherlands \\
\email{t.k.nhung.dang@gmail.com} \\
\and University of Twente, Enschede, the Netherlands  \\
\email{\{benedikt.peterseim, m.a.lopuhaa,  m.i.a.stoelinga\}@utwente.nl}\\
\and Radboud University, Nijmegen, the Netherlands \\
\email{m.stoelinga@cs.ru.nl}\\
}

\allowdisplaybreaks

\maketitle              

\begin{abstract}

We provide a rigorous framework for handling uncertainty in quantitative fault tree analysis based on fuzzy theory. 
We show that any algorithm 
for fault tree unreliability analysis can be adapted to this framework in a fully general and computationally efficient manner.
This result crucially leverages both the $\alpha$-cut representation of fuzzy numbers and the coherence property of fault trees.
We evaluate our algorithms on an established benchmark of synthetic fault trees, demonstrating their practical effectiveness.



\end{abstract}

\section{Introduction}
\label{sec:intro}

\emph{Fuzzy fault trees} are a tool to assess the dependability of safety-critical systems, while simultaneously quantifying the uncertainty that enters these models through their parameters. 
To achieve this, they aim to combine classical \emph{fault tree analysis} with concepts from \emph{fuzzy theory}.
Our goal is to make the idea of fuzzy fault tree analysis rigorous, and to provide computationally fast methods for carrying it out.
\vspace{-15pt}

\subsubsection{Fault trees.} Fault tree analysis (FTA) is a popular method in reliability engineering~\cite{stamatelatos2002fault,ruijters2015FTA}.
It is widely used in industry to assess and improve the dependability of, amongst others,
nuclear power plants, self-driving cars, and aeroplanes.
FTA is recommended by several ISO standards and certification bodies, such as the Federal Aviation Administration (FAA). 
A key aspect of FTA is quantitative assessment, calculating essential \emph{dependability} or \emph{risk metrics}, such as \emph{unreliability}, \emph{availability}, and \emph{mean time to failure}.
This paper studies the so-called mission-time reliability model \cite{stoelinga2025concise}. Here 
each basic event $b$
is assigned a probability $p_b$, representing  
its probability to fail within mission time. 
From these, one can compute the failure probability of the top event,
i.e.~the probability of the system to fail within its mission time, called the 
{\it system unreliability}.

\subsubsection{Fault trees under uncertainty.} Reliability analysis presupposes the availability of precisely known failure probabilities $p_b$. However, in practice this assumption may be unrealistic due to conflicting expert opinions, or the lack of reliable data. In such situations, \emph{uncertainty quantification} enables a precise assessment of the confidence in the estimated unreliability. While there are other approaches to uncertainty quantification in fault trees, as discussed in \Cref{sec:related_works}, we will focus on devising precise and efficient foundations for an approach rooted in \emph{fuzzy theory}.

%

\subsubsection{Fuzzy theory.}\label{par:fuzzy_theory_explanation} 
Fuzzy theory has successfully been applied in numerous domains, 
including control systems~\cite{deBarros2017afirst}, medical imaging, economic risk assessment, decision trees~\cite{Basiura2015advances}, and machine learning~\cite{couso2019fuzzy}. 
Its application to fault trees yields \emph{fuzzy fault trees}.
These handle parameter uncertainty by taking the failure probabilities of basic events to be \emph{fuzzy numbers}.  
Whereas most works on fuzzy fault trees are case studies with an emphasis \emph{how} fuzzy probabilities of basic events are obtained in practice (see, for example, \cite{yazdi2019uncertainty}), we assume that these basic fuzzy probabilities are given. 
Instead, our focus will be to rigorously define the \emph{fuzzy unreliability} in a principled way, and how to compute it efficiently.

\subsubsection{Challenges.}\label{par:challenges}
Despite its successful application in many case studies (see \Cref{sec:related_works}), fuzzy FTA still lacks a rigorous mathematical foundation. Unclear or ad-hoc definitions of system unreliability in fuzzy fault trees impede the interpretability of the resulting risk metric and are hence not an acceptable means for decision-making in safety-critical situations.
In addition, to the best of our knowledge, no generally applicable, precise and efficient algorithm for quantitative fuzzy FTA has so far been presented. 
This paper aims to close both of these gaps.

One major obstacle in fuzzy fault tree analysis is that performing exact arithmetic operations on fuzzy numbers is generally computationally expensive.
Most notably, common classes, or ``shapes'', of fuzzy numbers such as \emph{triangular} and \emph{trapezoidal} fuzzy numbers are not closed under basic (``fuzzified'') arithmetic operations~\cite{tanaka1983fault,Liang1993FFTA,Basiura2015advances}.
For example, 
if we multiply two triangular fuzzy numbers
(via the canonical \emph{Zadeh extension}), then the result is no longer triangular~\cite{Liang1993FFTA}. 



\subsubsection{Contributions.}
To overcome these challenges, this paper contributes:
\begin{enumerate}
    \item A well-motivated, principled and mathematically rigorous definition of the \emph{fuzzy unreliability} metric;
    \item A fast bottom-up algorithm based on \emph{$\alpha$-cuts} for computing fuzzy unreliability in \emph{tree-structured} fault trees in a simple and intuitive way;
    \item A correctness result showing that, even in the general case of \emph{DAG-structured} FTs, \emph{any} unreliability algorithm can be extended to the fuzzy case, assuming a mild regularity condition on fuzzy numbers that holds for all classes of fuzzy numbers used in practice; see \Cref{theo:DAG-algo};
    \item An empirical evaluation of our algorithms using the model checker Storm~\cite{Hensel2022storm}.
\end{enumerate}

Our main result, \Cref{theo:DAG-algo}, enables a simple implementation of unreliability analysis in fuzzy fault trees using existing tools, making our algorithms readily applicable in practice. 
The reason this works is a delicate interplay between two crucial assumptions underlying both fault tree analysis and fuzzy theory. 
The main property of fault trees used is that they are \emph{coherent}: the failure of any basic event never \emph{decreases} the failure probability of the top event. 
On the other hand, the main common assumption on fuzzy numbers we use is that their $\alpha$-cuts (i.e.~$\alpha$-upper level sets) are intervals. 
The insight that fuzzy probabilities are faithfully and efficiently represented and computed by their $\alpha$-cuts is the final ingredient which makes our methods work.


\section{Fault trees}\label{sec:FTs}

\begin{wrapfigure}[16]{r}{0.4\textwidth}
\centering
\vspace{-6em}
\includegraphics[width=4.5cm]{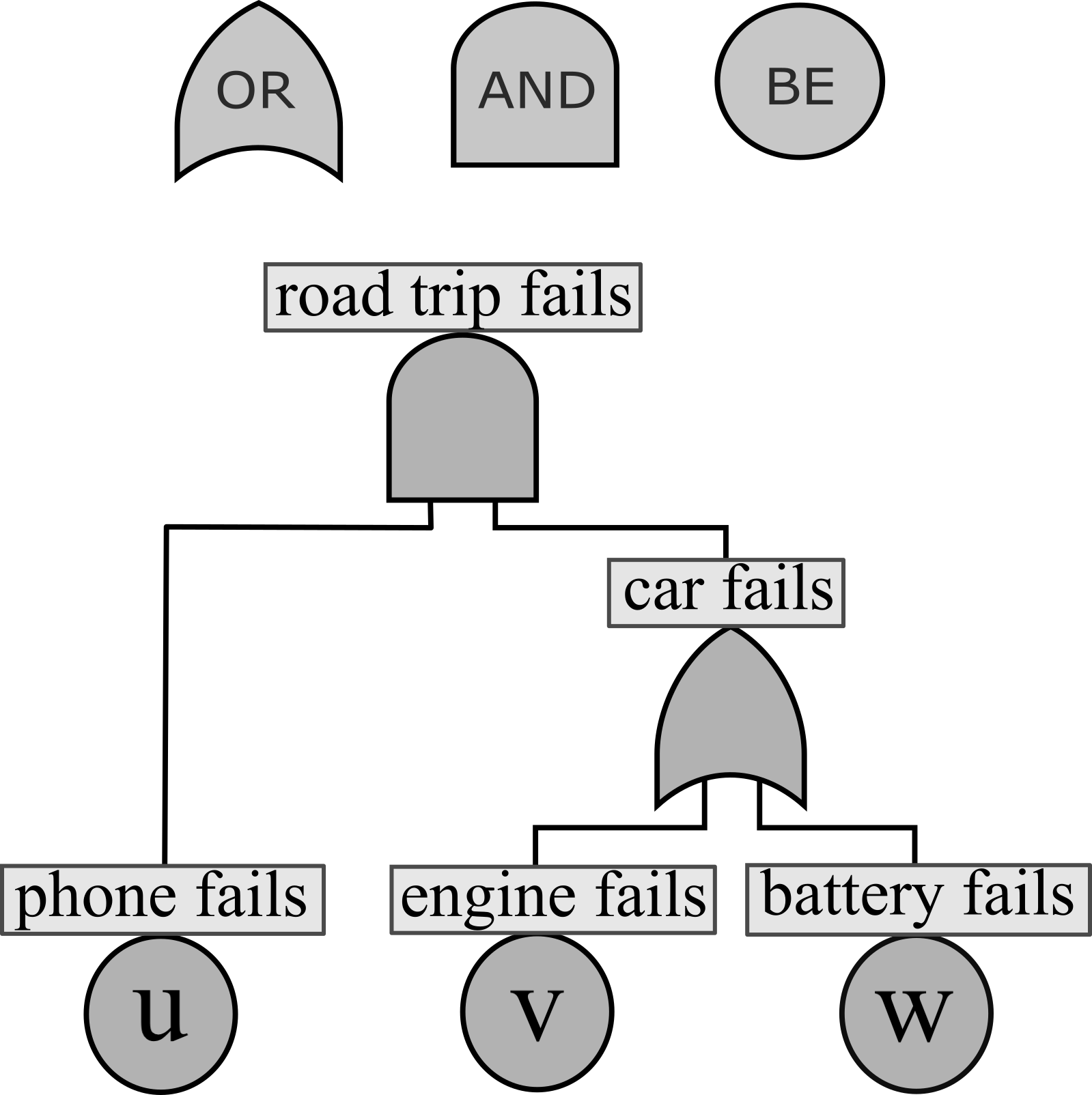}
\caption{Fault tree for a road trip. The road trip fails if both the phone fails and the car fails; the latter happens when either the engine or the battery fails.
Its structure function equals
$S_T(\,\vec{b}\:) = b_u \wedge (b_v \vee b_w)$.
}
\label{fig:FT-intro}
\end{wrapfigure}

Fault trees (FTs) are hierarchical diagrams whose top event represents system failure, and whose leaves, called \emph{basic events} (BEs), represent atomic failures. 
Intermediate gates are AND- or OR-gates, and propagate failures according to the status of their inputs; see Fig.~1.


\begin{definition}\label{def:FTs}
    A \emph{fault tree} (FT) is a tuple $T=(V,E,t)$, where $(V,E)$ is a rooted directed acyclic graph (DAG), and $t$ is a map $t\colon V \rightarrow \{\mathtt{BE}, \mathtt{OR}, \mathtt{AND}\}$ such that for all $v \in V$, $t(v) = \mathtt{BE}$ if and only if $v$ is a leaf. 
    The set of basic events is written $\basicevents_T = \{ v\in V | t(v) = \mathtt{BE}\}$. 
    The root of $T$ is denoted $R_T$. For a node $v \in V,$ we write $ch(v)$ for the set of all children of $v$.
\end{definition}
A fault tree $T = (V, E, t)$ need not be a tree in the graph-theoretic sense. If the underlying DAG $(V, E)$ forms a tree (each node has a unique parent), it is \emph{tree-structured}; otherwise, it is \emph{DAG-structured}.

A binary vector $\vec{b} \in \mathbb{B}^{\text{BE}_T}$ is called a \emph{status vector}, where $\mathbb{B}=\{0,1\}$ is the set of Booleans, where 1 indicates failure, and 0 means operational. 
A \emph{probabilistic status vector} is a probability vector $\vec{p} \in [0,1]^{\basicevents_T}$.
Whether or not the overall system fails given a status vector is determined by the \emph{structure function}: 

\begin{definition}
\label{def:structure_func}
Let $T$ be a FT. The \emph{structure function} $S_T\colon V \times \mathbb{B}^{\text{BE}_T} \to \mathbb{B}$ of $T$ is defined, for a node $v \in V$ and a status vector $\vec{b}=(b_w)_{w\in \text{BE}_T}\in \mathbb{B}^{\text{BE}_T}$, by
\begin{align*} 
S_T(v,\vec{b}) =&
\begin{cases}
    \bigvee_{w\in ch(v)}S_T(w,\vec{b})  & \parbox{55pt}{if~$t(v)=\mathtt{OR}$,}\\
    \bigwedge_{w\in ch(v)}S_T(w,\vec{b})  & \parbox{55pt}{if~$t(v)=\mathtt{AND}$},\\
    b_v  & \parbox{55pt}{if~$t(v)=\mathtt{BE}$.}
\end{cases}
\end{align*}
We write $S_T(\,\vec{b}\:):= S_T(R_T, \vec{b})$ for the structure function of $T$ at its roots. 
A status vector $\vec{b}$ that reaches the root $R_T$ i.e., 
$S_T(\,\vec{b}\:)=1$ is called a \emph{cut set}. 
The set of all cut sets of $T$ is denoted $\mathcal{C}_T$.
\end{definition}

\subsection{Fault tree reliability analysis}


The \emph{system unreliability} is the probability that a system fails within its mission time. 
In fault tree analysis, the unreliability is obtained as the probability that the top event occurs, if each basic event $v$ is assigned a failure probability $p_v$, i.e. the probability that this BE fails within its mission time. 
The latter are given as a \emph{probabilistic status vector} $\vec{p} \in [0,1]^{\text{BE}_T}$.
%

 

\begin{definition}\label{def:unreliability}
    Let $T=(V,E,t)$ be a FT with probabilistic status vector $\vec{p}=(p_v)_{v\in V}$. The \emph{unreliability} of $T$ with respect to $\vec{p}$ is defined as
    $$ U_T(\vec p) := \mathbb{P}\left[S_T\left(\,\vec{B}\,\right) = 1\right],$$
    where $\vec{B} = (B_v)_{v \in \text{BE}_T}$ is a random vector whose components $B_v$ are all independent and Bernoulli-distributed with probability $p_v$.
\end{definition}

This definition assumes all basic event failures to be independent---a standard assumption, since all dependencies are captured by the FT gates. 
Using the definition of $\mathcal{C}_T$, we see that $U_T(\vec{p})$ is equivalently given by,
\begin{align}\label{eq:UT}
    U_T(\vec{p}) &= \sum_{\vec{b} \in \mathcal{C}_T} \prod_{v \in \mathrm{BE}_T} p_v^{b_v} \cdot \bigl(1-p_v\bigr)^{(1-b_v)}.
\end{align}




\section{Fuzzy numbers}\label{sec:fundamentals_fuzzy}

\begin{wrapfigure}[10]{r}{0.4\textwidth}
\begin{center}
\vspace{-3.5em}
\includegraphics[width=0.8\textwidth]{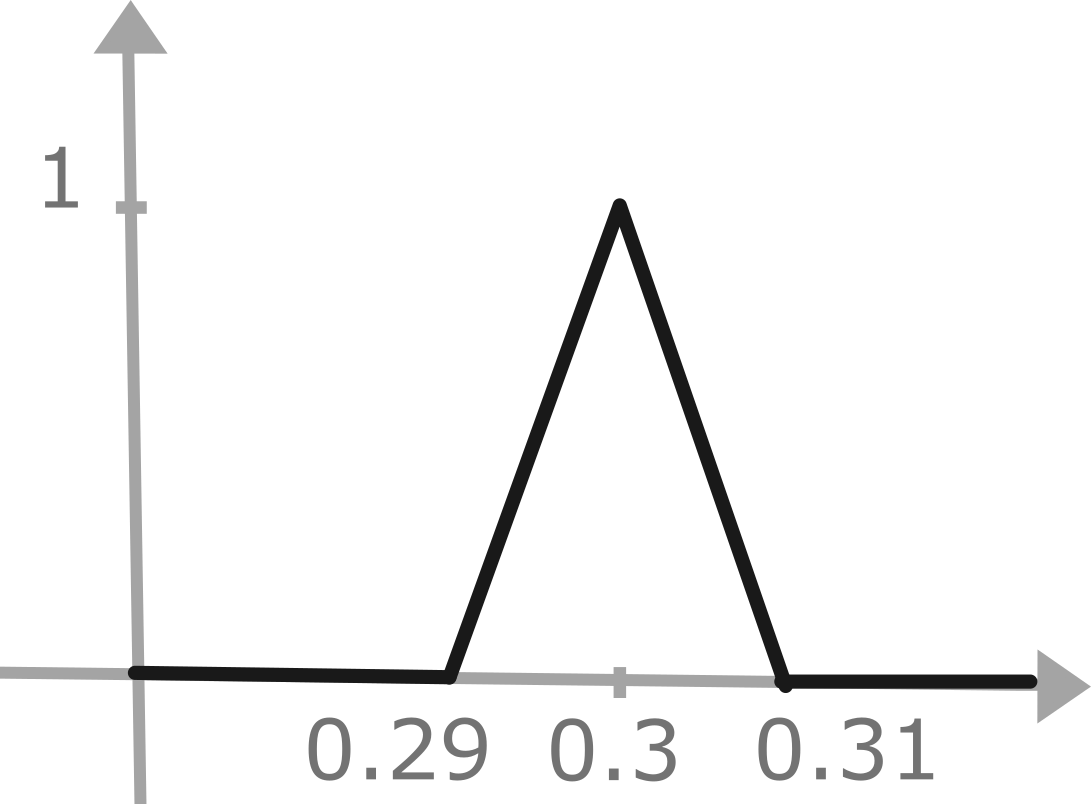}
\vspace{-0.5em}
\caption{``Approximately $0.3$''}
\label{fig:fuzzy-ex}
\end{center}
\end{wrapfigure}

Fuzzy theory was proposed in  \cite{zadeh1965fuzzy}
to reason about vagueness in a precise way.
\Cref{fig:fuzzy-ex} illustrates the 
fuzzy number $\mathsf{x}= $ ``approximately $0.3$''. 
Here, 
 $\mathsf{x}$ is not a single value, but rather a function, called the  \emph{fuzzy membership degree}.
That is,
$\mathsf{x}[x]$ indicates
\emph{how much} $x$ resembles $0.3$:
 At $0.3$, the membership degree is $1$, indicating that ``$0.3$ is unequivocally \emph{$0.3$}''. As we move farther from $0.3$, the likeness to ``approximately $0.3$'' diminishes, and at $0.31$ or $0.29$, it is clear these numbers are \emph{not} ``approximately 0.3''. 
Alternatively, membership functions can represent trust, where $\mathsf{x}[x]$ denotes our trust in $\mathsf{x}$ equalling $x$. If $\mathsf{x}[x] = 0$, then trust is zero; if $\mathsf{x}[x] = 1$, trust is maximal. This notion can be expressed for members of any set, leading to the definition of \emph{fuzzy elements}.



\begin{definition}
A \emph{fuzzy element} of a set $X$ is a function $X \rightarrow [0,1]$. The set of 
fuzzy elements of $X$ is denoted $\mathbf{F}(X)$. A \emph{fuzzy number} is a fuzzy element of $\RR$.
\end{definition}

\myparagraph{Classes of fuzzy numbers.}\label{sec:classes-of-fuzzy-numbers} Several common types of fuzzy numbers exist. 
For real numbers $a \leq b \leq c \leq d$, the \emph{trapezoidal fuzzy number} $\mathsf{trap}_{a,b,c,d} \in \mathbf{F}(\mathbb{R})$  is defined as (see Fig.~\ref{fig:fuzzy-number-plots}):

\begin{align*}
    \mathsf{trap}_{a,b,c,d}[x] &:= 
    \begin{cases}
        \tfrac{x-a}{b-a}, & \textrm{if } a < x < b,\\
       1, & \textrm{if } b\leq x \leq c,\\
        \tfrac{d-x}{d-c}, & \textrm{if } c < x < d,\\
        0, & \text{otherwise}.
    \end{cases}
\end{align*}
Trapezoidal fuzzy numbers generalize \emph{triangular fuzzy numbers}, which are defined as $\triangle_{a,b,d}$ $:=$ 
$\mathsf{trap}_{a,b,b,d}$, 
and interval fuzzy numbers ($\mathbb{1}_{[a,b]} := \mathsf{trap}_{a,a,b,b}$). 
The latter allows for the treatment of \emph{imprecise probability} in fault trees \cite{jacob2011uncertainty,jacob2012imprecise} as a special case of fuzzy fault tree analysis.
\emph{Gaussian fuzzy numbers} are characterized by their mean $m$ and standard deviation $d$:
\[\mathsf{gauss}_{m,d}[x] := \exp{(\tfrac{-(x-m)^2}{2d^2})}.\]
\begin{figure}[t]
    \centering
    \includegraphics[width=1.0\linewidth]{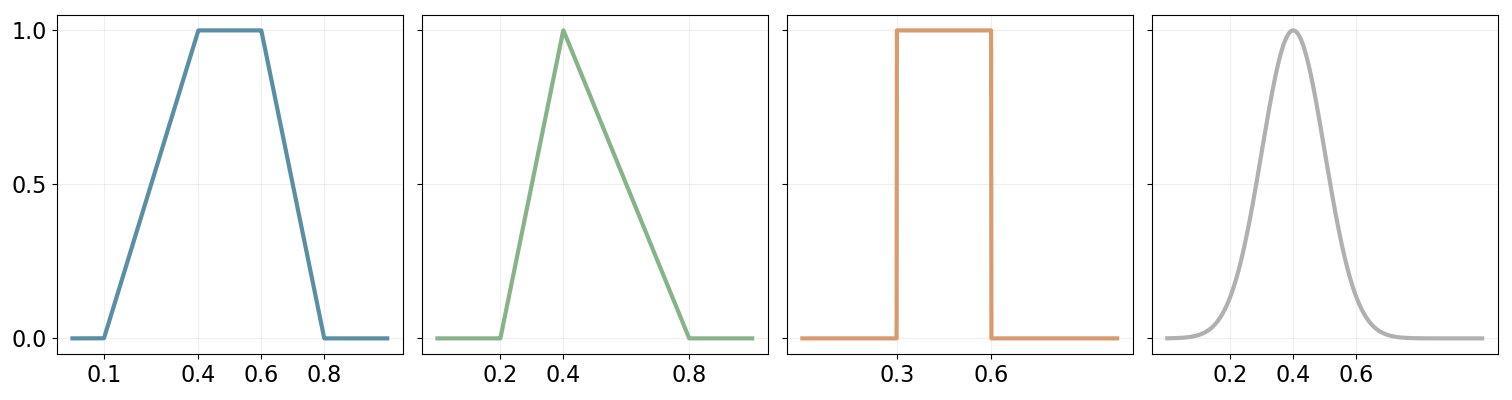}
    \caption{Membership functions, from left to right, of a \emph{trapezoidal} $\mathsf{trap}_{0.1,\, 0.4, \,0.6, \,0.8}$, \emph{triangular} $\triangle_{0.2, \,0.4,\,0.8}$, \emph{interval} $\mathbb{1}_{[0.3, \,0.6]}$, and \emph{Gaussian} $\mathsf{gauss}_{0.4,\,0.1}$ fuzzy number.}
    \label{fig:fuzzy-number-plots}
\end{figure}

%



\subsection{Zadeh's extension principle}

Zadeh's extension principle \cite{Jezewski2017theory,zadeh1965fuzzy} lifts any function $f:X\to Y$ to a function of fuzzy elements 
$\widetilde{f}\colon \mathbf{F}(X)\rightarrow \mathbf{F}(Y)$.
To understand how this works, assume first that $f$ is injective and that $y\in Y$ with $f(x)=y$. 
Then we set $\widetilde{f}(\mathsf{x})[y]= \mathsf{x}[x]$, as our trust for $f$ to be equal to $y$ at $\mathsf{x}$ should be the same as our trust for $\mathsf{x}$ to be equal to $x$. 
Also, for $y\in Y$ with $f^{-1}(y) = \emptyset$, we set $\widetilde{f}(\mathsf{x})[y]=0$, since $f$ never takes on the value $y$.
Now, if $f$ is not injective, then there are multiple values $x$ with $f(x)=y$. Each of these values $x\in f^{-1}(y)$ has a fuzzy membership degree $\mathsf{x}[x]$. Zadeh's extension principle takes the highest possible membership degree among these. The underlying idea is as follows: 
Consider a fuzzy number $\mathsf{x}$ with
$\mathsf{x}[-2]= 0.3$ and 
$\mathsf{x}[2]= 0.9$
consider the function $f(x)= x^2$. 
How much trust do we have that ``$\mathsf{x}^2=4$'', i.e.  what is the value of $\widetilde f(\mathsf{x})[4]$? 
We can defend a trust degree of $0.9$, by assuming that the value $\mathsf{x}^2=4$ was obtained by taking as input $x=2$, which has trust degree $\mathsf{x}[2]= 0.9$.

Finally, for functions of multiple arguments, membership degrees of independent arguments are combined by taking their minimum.
This is justified as follows: to trust in the value of a multivariate function $f$ with a degree of $\alpha$, our trust at \emph{all} of its arguments must be at least $\alpha$.

\begin{definition}[Zadeh's Extension Principle]\label{def:extension_principle}
Let $f:X_1\times\dots \times X_n \to Y$ be a function. The \emph{Zadeh extension} of $f$ is defined as the function, 
\begin{align*}
 \widetilde{f} & \colon \mathbf{F}(X_1) \times \dots \times \mathbf{F}(X_n) \rightarrow \mathbf{F}(Y),  \\ 
   \widetilde{f}(\vec{\mathsf{x}})[y]
& := 
\sup \left\{\left.\min_{i=1,\dots n} \mathsf{x}_i[x_i] \:\right\vert\: \vec x \in f^{-1}(y) \right\}, 
\end{align*}
for all $y\in Y$, $\vec{\mathsf{x}} \in \mathbf{F}(X_1) \times \dots \times \mathbf{F}(X_n)$.
\end{definition}


On certain families of fuzzy elements, addition and subtraction operations can be performed in a straightforward manner. For example, for two trapezoidal fuzzy numbers we have
\begin{align*}
    \mathsf{trap}_{a_1,a_2,a_3,a_4} \  \widetilde{+} \ \mathsf{trap}_{b_1,b_2,b_3,b_4} &= \mathsf{trap}_{a_1 + b_1, a_2 + b_2, a_3 + b_3, a_4 + b_4},\\
    \mathsf{trap}_{a_1,a_2,a_3,a_4} \ \widetilde{-} \ \mathsf{trap}_{b_1,b_2,b_3,b_4} &= \mathsf{trap}_{a_1 - b_4, a_2 - b_3, a_3 - b_2, a_4 - b_1}.
\end{align*}

In general, however, there are no such simple formulas for the Zadeh extension of arithmetic operations. In particular, arithmetic operations do not preserve the shape of the fuzzy numbers:
For example, the product of two trapezoidal fuzzy numbers is \emph{not} a trapezoidal fuzzy number, in general.
%

\section{Fuzzy fault trees}\label{sec:fuzzy_unreliability}

Fuzzy failure probabilities arise when the failure probabilities of basic events as fuzzy numbers. Then, the fuzzy unreliability is obtained by the Zadeh extension of the system unreliability function.

\begin{example}\label{ex:fuzzy_unreliability}
Consider again the FT $T$ from Fig.~\ref{fig:FT-intro}. By \Cref{eq:UT} its unreliability function $U_T\colon [0,1]^3 \rightarrow [0,1]$ is given by
\begin{align*}
    U_T(p_a, p_b, p_c) &= p_ap_bp_c+  p_a(1-p_b)p_c +  p_ap_b(1-p_c)\\
    &= p_ap_c+p_ap_b-p_ap_bp_c.
\end{align*}
Now, suppose $T$ is equipped with a \emph{fuzzy} failure probabilistic status vector $\vec{\mathsf{p}}=(\mathsf{p}_a, \mathsf{p}_b, \mathsf{p}_c)$. 
The fuzzy unreliability $\widetilde{U}_T(\mathsf{p}_a, \mathsf{p}_b, \mathsf{p}_c)$ is then defined as the Zadeh extension of $U_T$. 
\[
\widetilde{U}_T(\vec{\mathsf{p}})[y] = \sup_{p_a,p_b, p_c}  \min \{ \mathsf{p}_a[p_a], \mathsf{p}_b[p_b], \mathsf{p}_c[p_c] \} ,
\]
where the supremum is taken over the set of all $p_a,p_b, p_c \in [0,1]$ such that $p_ap_c+p_ap_b-p_ap_bp_c  = y$,  for all $y\in [0,1]$.

Suppose that $\mathsf{p}_b$ and $\mathsf{p}_c$ are crisp numbers with $\mathsf{p}_b=0.1$ and $\mathsf{p}_c=0.4$. Assume that $\mathsf{p}_a$takes probability $0.5$ or $0.8$ with fuzzy membership values $0.7$ and $1$, respectively. 
%
\begin{align*}                 
\widetilde{U}_T(\vec{\mathsf{p}}) =\begin{cases}
                        1, & \textrm{ if $y=0.368$},\\
                        0.7, & \textrm{ if $y=0.23$},\\
                        0, & \textrm{ otherwise}.
                    \end{cases}
    \end{align*}
\end{example}


\noindent Generalising this example, we obtain our main definition.

\begin{definition}[Fuzzy unreliability, fuzzy fault trees]\label{def:fuzzy_unreliability}
    Let $T$ be a FT. 
    \begin{enumerate}
            \item A \emph{fuzzy probabilistic status vector} is an element $\vec{\mathsf{p}}$ of $\mathbf{F}([0,1])^{\mathrm{BE}_T}$.

            \item The \emph{fuzzy unreliability} of $T$ given $\vec{\mathsf{p}}$ is defined as $\widetilde{U}_T(\vec{\mathsf{p}})$, where 
            $$\widetilde{U}_T\colon \mathbf{F}([0,1])^{\mathrm{BE}_T} \rightarrow \mathbf{F}([0,1])$$ 
            is the Zadeh extension of the function $U_T$ from Definition~\ref{def:unreliability}.
    \end{enumerate}
More concretely, $\widetilde{U}_T(\vec{\mathsf{p}})$ is the fuzzy element of $[0,1]$ defined by,
\begin{align}\label{eq:fuzzy_unreliability}
\widetilde{U}_T(\vec{\mathsf{p}})[y] &= \sup \left\{ \left.\min_{\;v \,\in\, \mathrm{BE}_T} \mathsf{p}_v[p_v] \;\right|\; \vec{p} \in [0,1]^{\mathrm{BE}_T},\; U_T(\vec{p}) = y
\right\}, 
\end{align}
for all $y\in [0,1]$. A \emph{fuzzy fault tree} is a fault tree equipped with a fuzzy probabilistic status vector.
\end{definition}

\begin{wrapfigure}[12]{r}{0.45\linewidth}
    \centering
    \vspace{-35pt}
    \includegraphics[width=0.95\linewidth]{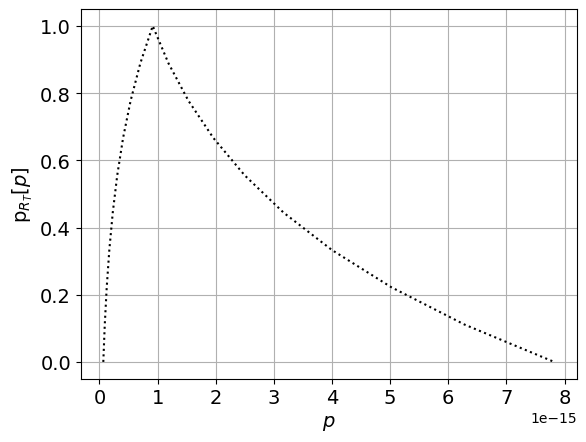}
    \caption{Fuzzy unreliability $\widetilde{U}_T(\vec{\mathsf{p}})$ for specific $T$ and $\vec{\mathsf{p}}$ from \Cref{sec:dag-structured-experiments}.}
    \label{fig:top_p_random}
\end{wrapfigure}

We will see a concrete, illustrative computation of the fuzzy unreliability in \Cref{ex:fuzzy_unreliability_BU_calc}. 
More realistic examples will be provided in the experiments in \Cref{sec:experiment}. 
Figure~\ref{fig:top_p_random} illustrates the fuzzy probability of top-level event failure for a particular fuzzy fault tree $(T, \vec{\mathsf{p}})$, which we have randomly selected from the benchmark used in \Cref{sec:dag-structured-experiments}. From the figure, we see that the unreliability shows a considerable skew and non-linearity, despite all basic events being equipped with symmetric triangular fuzzy probabilities. We hence obtain a much more informative and complete picture of the uncertainty in the probability of top-level system failure, going beyond both point-estimates and probability intervals.

\section{Computing fuzzy unreliability I: tree-structured case}\label{sec:tree-str-calc}

When $T$ is tree-structured, the fuzzy unreliability $\tilde{U}_T(\vec{\mathsf{p}})$ can be found using a bottom-up algorithm that proceeds exactly as for ordinary fault trees, replacing arithmetic operations by their Zadeh extensions.
We first review the ordinary, ``crisp'' case.
The general DAG-structured case will be treated in \Cref{sec:DAG_FTs}.
\subsubsection{Crisp case.}\label{sec:bottom-up-ordinary-fts}
Already in the crisp case, computing the unreliability for a fault tree is generally difficult (in fact, NP-hard~\cite{lopuhaazwakenberg2023ftreliability}). 
Naively applying \Cref{eq:UT} requires a summation over the entire set $\mathcal{C}_T$ of cut sets, and is therefore computationally infeasible for large FTs. 
When $T$ is tree-structured, the unreliability can instead be computed in a bottom-up fashion, assigning a probability $p_v$ to each node along the way. 
For a node $v$ with children $v_1,\ldots,v_n$, we write
\begin{align} 
p_v :=&
\begin{cases}\label{eq:prob_laws}
    1 - \prod_{i=1}^{n} (1 - p_{v_i})  & \parbox{55pt}{if~$t(v)=\mathtt{OR}$,}\\
    \prod_{i=1}^{n} p_{v_i}  & \parbox{55pt}{if~$t(v)=\mathtt{AND}$}.
\end{cases} 
\end{align}
Then, if $T$ is tree-structured, $U(T)=p_{R_T}$.
\begin{example}\label{ex:compute_probability}
    Consider the FT from Fig.~1; see also below. Let $p_u=0.8, p_v=0.1$, and $p_w=0.4$. The top-level failure probability is    
    \begin{align*}
        p_{R_T} &= p_u \cdot \bigl(1 - (1 - p_v) \cdot (1 - p_w) \bigr) = 0.368.
    \end{align*}
\end{example}

\begin{wrapfigure}[7]{r}{0.3\textwidth}
\centering
\vspace{-3em}
\includegraphics[width=3cm]{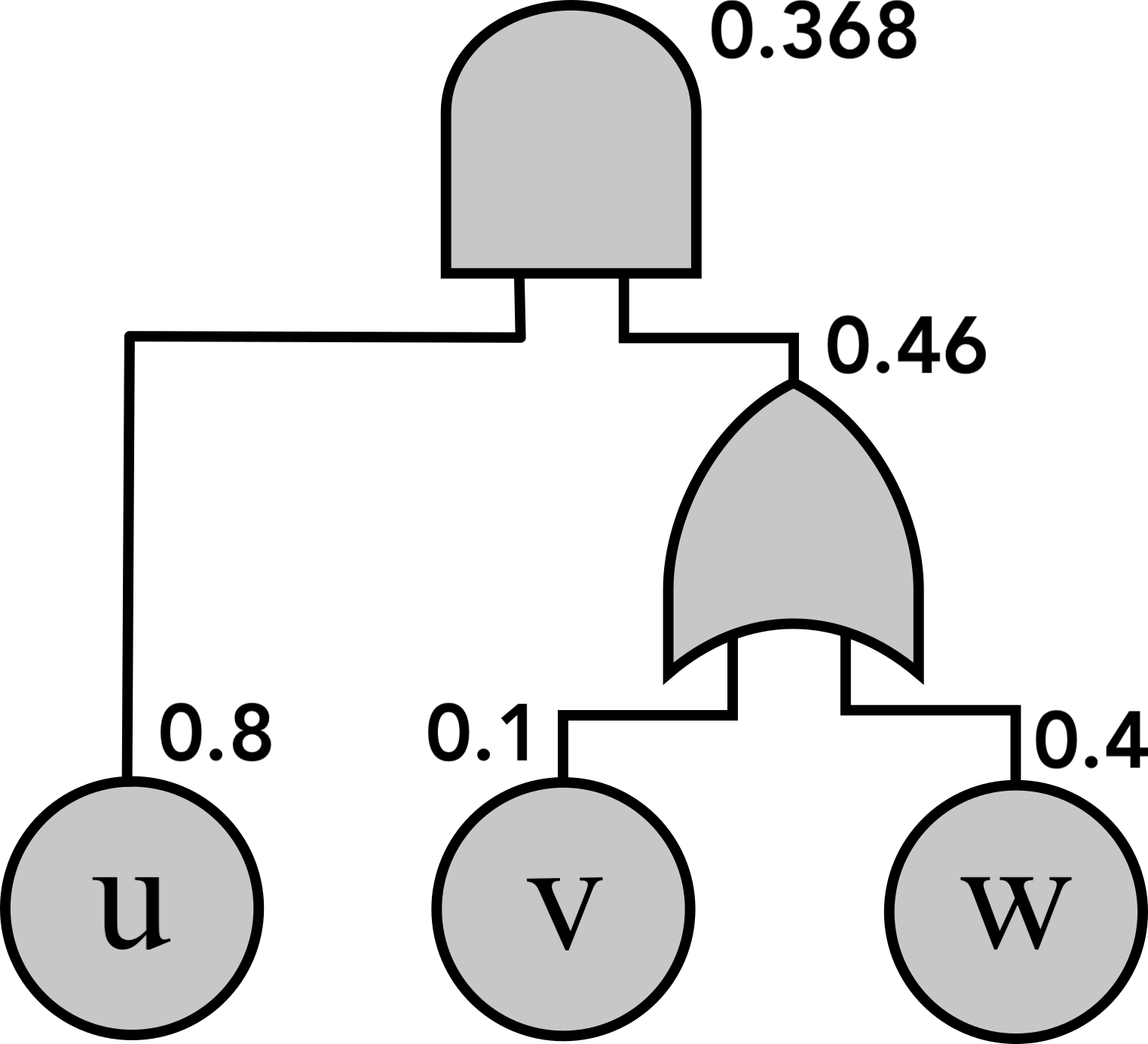}
\end{wrapfigure}

The bottom-up algorithm is fast, but does not extend to general, DAG-structured FTs. 
The reason is that \eqref{eq:prob_laws} only computes the failure probability of $v$ correctly when the $v_i$ all represent independent events, which will not generally be true if they share children. For DAG-structured FTs the state-of-the-art approach is to translate the FT to a binary decision diagram \cite{rauzy1993new}, on which a bottom-up algorithm is run. The worst-case time complexity of this approach is exponential, but is very fast in practice \cite{basgoze2022BDDs}.

\subsubsection{Fuzzy case.} To compute the \emph{fuzzy} unreliability, we now replace arithmetic operations by their Zadeh extensions, letting
\begin{equation} \label{eq:bu}
\mathsf{p}_{v} =
\begin{cases}
    1 \; \widetilde{-} \; \widetilde{\prod}_{w \in ch(v)} (1 \;\widetilde{-}\; \mathsf{p}_{w})  & \parbox{55pt}{if~$t(v)=\mathtt{OR}$,}\\
    \widetilde{\prod}_{w \in ch(v)} \mathsf{p}_{w} & \parbox{55pt}{if~$t(v)=\mathtt{AND}$}.
\end{cases} 
\end{equation}
Here, we write ``$1$'' for the (``crisp'') fuzzy number whose membership function is $1$ at the number $1$ and vanishes everywhere else.
The following result then states that the fuzzy unreliability can be computed recursively using \Cref{eq:bu}. It is proved analogously to a similar result for attack trees in~\cite{dang2024fuzzy}.

\begin{theorem}
Let $T$ be a tree-structured FT, and let $\vec{\mathsf{p}} \in \mathbb{F}([0,1])^{\text{BE}_T}$ be a vector of fuzzy probabilities. Then $\mathsf{p}_{R_T} = \widetilde{U}_T(\vec{\mathsf{p}})$.
\end{theorem}

In the subsequent illustrative example, we employ the following compact notation for fuzzy probabilities that take only finitely many values,
$$
\{x_1\mapsto a_1, \dots, x_n\mapsto a_n\}[x] := 
\begin{cases}
    a_i \quad \text{ if } x = x_i \text{ for some } i\in \{1,\dots, n\}, \\
    0 \quad \text{ otherwise,}
\end{cases}
$$
for any $x_1, \dots, x_n, a_1, \dots, a_n, x \in [0,1]$ with $x_1, \dots, x_n$ distinct.

\begin{example}\label{ex:fuzzy_unreliability_BU_calc}
    We apply the algorithm to Ex.~\ref{ex:fuzzy_unreliability}. Given 
    $$ \mathsf{p}_a :=  \{0.5 \mapsto 0.7, 0.8 \mapsto 1\}, \;
    \mathsf{p}_b :=  \{0.1 \mapsto 1\}, \;
    \mathsf{p}_c := \{0.4 \mapsto 1\}, $$
    we calculate the unreliability as follows:
    \begin{align*}
    \mathsf{p}_{\texttt{OR}(b,c)} &= 1\,\widetilde{-}\,(1\,\widetilde{-}\,\mathsf{p}_b)\:\tilde{\cdot}\:(1\,\widetilde{-}\,\mathsf{p}_c) \\
    &= 1\,\widetilde{-}\, \{0.9 \mapsto 1\} \:\tilde{\cdot}\: \{0.6 \mapsto 1\} \\
    &= \{0.46 \mapsto 1\},
    \end{align*}
    and similarly for $\mathsf{p}_{R_T}=\mathsf{p}_{a}\:\tilde{\cdot}\:\mathsf{p}_{\texttt{OR}(b,c)}$.
    This way, we obtain 
    $$\tilde{U}_T(\vec{\mathsf{p}}) = \mathsf{p}_{R_T}= \{0.23 \mapsto 0.7, 0.368 \mapsto 1\}.$$ 
\end{example}

\section{Computing fuzzy unreliability II: general case}\label{sec:DAG_FTs}

When a fault tree is not tree-structured, but a general DAG-structured FT, then the bottom-up approach no longer computes $\tilde{U}_T(\vec{\mathsf{p}})$ correctly. This is because the probabilities of the children of a node may no longer be independent. In fact, this problem already arises for crisp probability values~\cite{ruijters2015FTA}, so it is no surprise that the same holds in the fuzzy setting.

For general DAG-structured fault trees, we instead leverage the $\alpha$-cut representation of fuzzy numbers.

\subsection{The $\alpha$-cut representation of fuzzy numbers}\label{sec:alpha-cuts}

Using fuzzy membership functions that take only finitely many non-zero values to represent fuzzy probabilities, as in \Cref{ex:fuzzy_unreliability_BU_calc}, is conceptually simple. 
However, this representation is inefficient for computing the Zadeh extension. This is because, in general, the number of non-zero values of the fuzzy unreliability's membership function will grow exponentially with the number of basic events.


If, instead, we choose to work with fuzzy numbers of a particular shape, such as trapezoidal fuzzy numbers, we face the issue that arithmetic operations do not preserve the shape of the fuzzy numbers.

One way to mitigate these problems
is to perform arithmetic operations on the given fuzzy numbers' \emph{$\alpha$-cuts}~\cite{Basiura2015advances}. 
The $\alpha$-cut of a fuzzy element $\mathsf{x}$ is the set of all elements of $X$ whose membership degree is at least $\alpha$.

\begin{figure}[h]
    \begin{center}
    \includegraphics[width=0.5\textwidth]{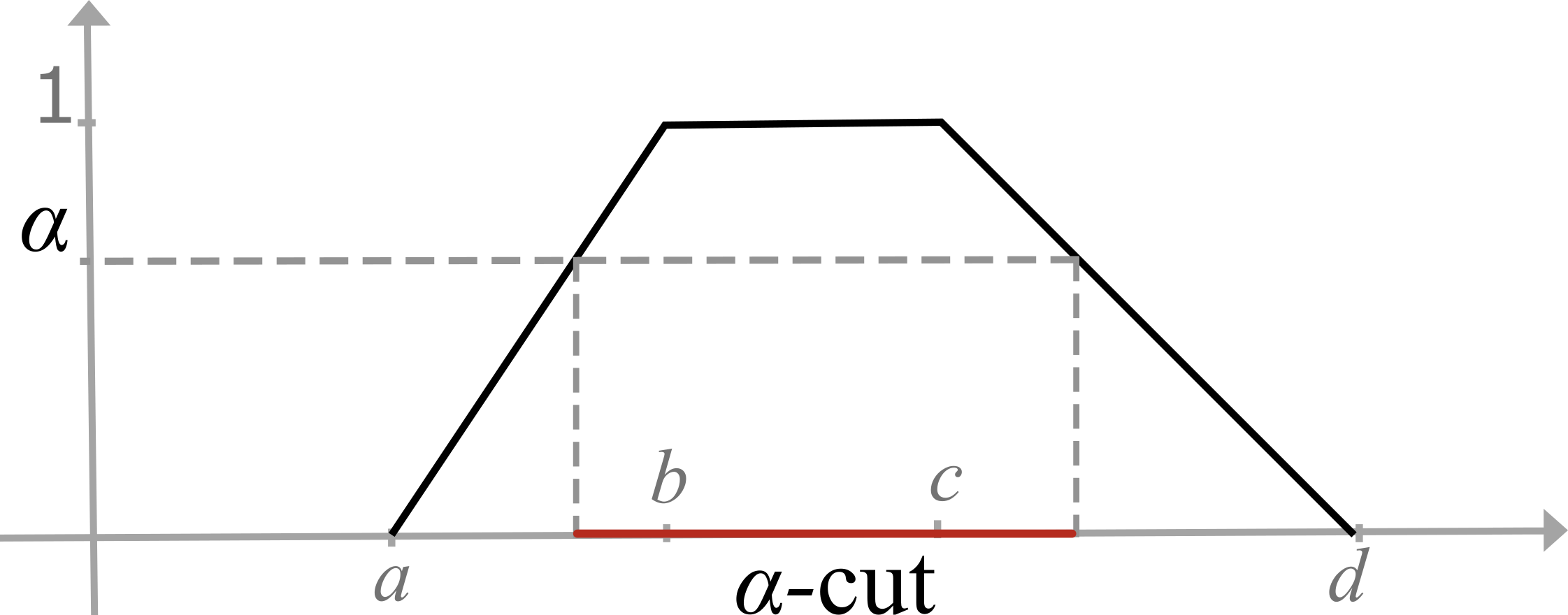}
    \end{center}
    \caption{$\alpha-$cut of trapezoidal fuzzy number $\mathsf{trap}_{a,b,c,d}$.}
    \label{fig:ex:trap_cuts}
\end{figure}

\begin{definition}\label{def:alpha_cut}
Let $\mathsf{x} \in \mathbf{F}(X)$ and $\alpha
\in [0,1]$. The \emph{$\alpha$-cut} $\mathsf{x}^{(\alpha)}$ of $\mathsf{x}$ is
\begin{equation}\label{eq:alpha_cut}
    \mathsf{x}^{(\alpha)} := \{x \in X \ | \ \mathsf{x}[x] \ge \alpha \}.
\end{equation}
\end{definition}

\begin{example}
The $\alpha$-cuts of a trapezoidal fuzzy number are given by

\begin{equation}\label{eq:alpha-cuts-of-triangular}
    \mathsf{trap}^{(\alpha)}_{a,b,c,d} = [(b-a)\cdot \alpha + a, \ d - (d-c)\cdot \alpha], 
\end{equation}

\noindent for all real numbers $a \leq b \leq c \leq d$ and all $\alpha\in (0,1]$ (see Fig.~\ref{fig:ex:trap_cuts}).
\end{example}

\myparagraph{Regular fuzzy numbers.} In the literature, the term ``fuzzy number'' almost universally refers to fuzzy elements of $\mathbb{R}$ that also satisfy certain regularity conditions, where the precise conditions vary across the literature  \cite{dijkman1983fuzzy}. 

We consider \emph{regular} fuzzy numbers, whose $\alpha$-cuts are all intervals. Regularity ensures that Zadeh extensions can be computed at the endpoints of these intervals, under mild assumptions; see \Cref{lemma:zadeh-extension-via-alpha-cuts}.

\begin{definition}\label{def:fuzzy-number}
    A \emph{regular fuzzy number} is a fuzzy element $\mathsf{x}\in\mathbf{F}(\mathbb{R})$ such that: 
    \begin{enumerate}
        \item For all $\alpha\in (0,1]$, the $\alpha$-cut $\mathsf{x}^{(\alpha)}$ is a compact interval, i.e.~there exist $l, r\in \mathbb{R}$ satisfying $\mathsf{x}^{(\alpha)} = [l,r]$. 
        \item The fuzzy membership function $\mathsf{x}: \mathbb{R} \to [0,1]$ is compactly supported, i.e.~the set of all $x\in \mathbb{R}$ with $\mathsf{x}[x]\neq 0$ is bounded. 
    \end{enumerate}
    We write $\fuzznum$ for the set of all regular fuzzy numbers.
    A fuzzy number $\mathsf{x}$ is \emph{nonnegative} if $\mathsf{x}[x]=0$ for all $x<0$. 
\end{definition}

Our definition is more general than e.g.~the one given in \cite{dubois1978operations} and covers most classes of fuzzy numbers used in practice.  
In particular, common parametrized
families of fuzzy numbers such as triangular, trapezoidal, and interval fuzzy numbers are fuzzy numbers in the above sense. 
Note that Gaussian fuzzy numbers are \emph{not} regular, as $\mathsf{gaus}_{m,d}[x] > 0$ for all $x \in \mathbb{R}$. 
Note, however, that since they represent fuzzy \emph{probabilities}, all fuzzy numbers appearing in fuzzy fault trees will have fuzzy membership functions vanishing outside of $[0,1]$, therefore automatically satisfying the second condition from \Cref{def:fuzzy-number}.

The regularity of a fuzzy number allows us to describe it in two alternative ways: either by its fuzzy  membership function $\mathsf{x}\colon \mathbb{R} \rightarrow [0,1]$, or by the two functions $(0,1] \rightarrow \mathbb{R}$ that assign to $\alpha$ the endpoints of the $\alpha$-cut of $\mathsf{x}$. 
The advantage of the second viewpoint is that Zadeh extensions become significantly easier to compute: it turns out that for nonnegative fuzzy numbers, Zadeh extensions of arithmetic operations can simply be computed at the endpoints of the $\alpha$-cuts. 

More generally, we have the following result, which we will later also use to compute the fuzzy unreliability of a fault tree.

\begin{lemma}\label{lemma:zadeh-extension-via-alpha-cuts}
    Let $I\subseteq \mathbb{R}$ be a closed interval, 
    let $\vec{\mathsf{x}}=(\mathsf{x}_1 , \dots, \mathsf{x}_n)\in \fuzznum^n$ be a tuple of regular fuzzy numbers each of whose fuzzy membership function is supported in $I$ (i.e. $\mathsf{x}[x]=0$ for all $x\not\in I$), 
    and let $f: I^n \to \mathbb{R}$ be a continuous function.
    Then:
    \begin{enumerate}
        \item For all $\alpha\in [0,1]$, the $\alpha$-cut of the Zadeh extension of $f$ is, 
        $$ 
        \widetilde{f}(\vec{\mathsf{x}})^{(\alpha)}
        = f\left[\mathsf{x}_1^{(\alpha)} \times \dots \times \mathsf{x}_n^{(\alpha)}\right],
        $$
        i.e.~the image of the product of the $\alpha$-cuts of each $\mathsf{x}_1 , \dots, \mathsf{x}_n$ under $f$.
        \item Assume, moreover, that $f$ is either monotonically non-decreasing or non-increasing, and write each $\alpha$-cut as an interval,
        $$
        \mathsf{x}_{i}^{(\alpha)} = 
        \left[\mathsf{x}_{i}^{(\alpha, l)}, \;\mathsf{x}_{i}^{(\alpha, r)}\right], 
        $$
        for all $i\in\{1,\dots, n\}, \,\alpha\in (0, 1]$. Then, for all $\alpha\in (0,1]$,
        $$
        \widetilde{f}(\vec{\mathsf{x}})^{(\alpha)}
        = \left[f\left(\vec{\mathsf{x}}^{\,(\alpha, l)}\right),
        \;f\left(\vec{\mathsf{x}}^{\,(\alpha, r)}\right)\right],
        $$ 
        where $\vec{\mathsf{x}}^{\,(\alpha, l)} = (\mathsf{x}_1^{(\alpha, l)} , \dots, \mathsf{x}_n^{(\alpha, l)})$
        and, similarly, $\vec{\mathsf{x}}^{\,(\alpha, r)} = (\mathsf{x}_1^{(\alpha, r)} , \dots, \mathsf{x}_n^{(\alpha, r)})$. In other words, the Zadeh extension of $f$ can be computed at the endpoints of each $\alpha$-cut.
    \end{enumerate}
\end{lemma}

The proof of \Cref{lemma:zadeh-extension-via-alpha-cuts} is given in \Cref{sec:proof-of-lemma-on-zadeh-extensions} of the appendix. Its first part is essentially well known \cite[Proposition 5.1]{nguyen1978note}, while the second part follows easily from the assumption of monotonicity. 

Applying \Cref{lemma:zadeh-extension-via-alpha-cuts} to the addition and multiplication functions $+,\cdot:$ $[0,\infty)^2$ $\to$ $[0,\infty)$ as well as to $x \mapsto 1-x$, we see that the Zadeh extension of all arithmetic operations that appear in elementary probability calculations can be computed on the level of $\alpha$-cuts. 

\begin{example}
    Consider two triangular fuzzy numbers $\triangle_{1,2,3}, \triangle_{3,4,6}\in \fuzznum$. We wish to calculate the (Zadeh-extended) product of these two fuzzy numbers using their $\alpha$-cuts. By \Cref{eq:alpha-cuts-of-triangular}, for any $\alpha \in (0,1]$,
    $$
    \triangle^{(\alpha)}_{1,2,3} = [1+\alpha, 3-\alpha], \quad
    \triangle^{(\alpha)}_{3,4,6} = [3+\alpha, 6-2\alpha].
    $$
    Therefore, by \Cref{lemma:zadeh-extension-via-alpha-cuts},  
    $$
    (\triangle_{1,2,3} \ \widetilde{\cdot} \ \triangle_{3,4,6})^{(\alpha)} = [\alpha^2 +4\alpha+3, 2\alpha^2-12\alpha+18]
    $$
    Clearly, these $\alpha$-cuts are not the ones of a triangular fuzzy number. To determine the fuzzy membership function of the above (Zadeh-extended) product of fuzzy numbers, we instead need to solve two nonlinear equations: $\alpha^2 +4\alpha+3=x$, and $2\alpha^2-12\alpha+18=x$. The solutions  in $[0,1]$ then yield the fuzzy membership function
    \begin{align}
    (\triangle_{1,2,3} \ \widetilde{\cdot} \ \triangle_{3,4,6})[x] &= 
    \begin{cases}
        -2 + \sqrt{1+x}, & \textrm{if } 3 \leq x \leq 8,\\
        \tfrac{6 - \sqrt{2x}}{2}, & \textrm{if } 8 < x \leq 18,\\
        0, & \text{otherwise}.
    \end{cases}
    \end{align}
    This provides an example of how working with $\alpha$-cuts tends to be simpler than working fuzzy membership functions directly.
\end{example}

\subsection{Computations via $\alpha$-cuts}\label{sec:BU_alpha_cuts}

To make our algorithms for computing the fuzzy unreliability applicable in practice, we first choose a finite number of $\alpha$-cuts $n_{\text{cuts}}$. We then represent a (regular) fuzzy number $\mathsf{x}$ by $n_{\text{cuts}}$ pairs $(\mathsf{x}^{(\alpha, l)}, \mathsf{x}^{(\alpha, r)})$ for each $\alpha \in \left\{\tfrac{1}{n_{\text{cuts}}},\tfrac{2}{n_{\text{cuts}}},\cdots,1\right\}$. The interpretation is that $\mathsf{x}^{(\alpha)} = [\mathsf{x}^{(\alpha, l)}, \mathsf{x}^{(\alpha, r)}]$. 
In this way, we can represent common families of fuzzy numbers such as triangular or trapezoidal fuzzy numbers at a good degree of accuracy in a finite array of a fixed size. Moreover, the fuzzy operations appearing in the bottom-up algorithm, see \eqref{eq:bu}, can now be performed level-wise using \Cref{lemma:zadeh-extension-via-alpha-cuts}. Hence, at each node, we need to make $\mathcal{O}(n_{\text{cuts}})$ crisp arithmetic operations, for a total time complexity of $\mathcal{O}(n_{\text{cuts}}\cdot |V|)$, where $|V|$ is the number of nodes of the given fault tree.

We will exploit $\alpha$-cuts in the same way in the general DAG-structured case.

\subsection{Computing fuzzy unreliability  via $\alpha$-cuts}

 In fact, the following result allows us to compute the $\alpha$-cuts of the fuzzy unreliability using \emph{any} algorithm for the unreliability of ordinary DAG-structured fault trees.

\begin{theorem}\label{theo:DAG-algo}
    Let $T$ be a fault tree and let $\vec{\mathsf{p}}=(\mathsf{p}_v) \in \mathbf{F}([0,1])^{\basicevents_T}$
    be a fuzzy probabilistic status vector. Moreover, assume that for all $v\in \basicevents_T$, $\mathsf{p}_v$ is a regular fuzzy number in the sense of \Cref{def:fuzzy-number} and write 
    $$
    \mathsf{p}_v^{(\alpha)} = [\mathsf{p}_v^{(\alpha, l)}, \mathsf{p}_v^{(\alpha, r)}], 
    $$
    as well as, $\vec{\mathsf{p}}^{(\alpha, l)}:= (\mathsf{p}_v^{(\alpha, l)})_{v\in \basicevents_T}$ and $\vec{\mathsf{p}}^{(\alpha, r)}:= (\mathsf{p}_v^{(\alpha, r)})_{v\in \basicevents_T}$, for all $\alpha\in (0,1]$.  
    Then 
    $$
    \widetilde{U}_T(\vec{\mathsf{p}})^{(\alpha)} = \left[U\left(\vec{\mathsf{p}}^{(\alpha, l)}\right),\: U\left(\vec{\mathsf{p}}^{(\alpha, r)}\right)\right],
    $$
    for all $\alpha\in (0,1]$.
\end{theorem}

In other words, the $\alpha$-cuts of the fuzzy unreliability can be computed at the endpoints of the intervals that constitute $\alpha$-cuts of the given fuzzy probabilistic status vector. The key reason for this result is the monotonicity of the unreliability function; a full proof of \Cref{theo:DAG-algo} is given in \Cref{sec:proof-DAG-algo} of the appendix. 


\subsubsection{Algorithms for fuzzy unreliability in the general case.}\label{sec:bdd-based-algo} 
As a direct consequence of \Cref{theo:DAG-algo}, 
we can adapt any unreliability algorithm (see, for example, \cite{rauzy1993new}) to also compute the fuzzy unreliability $\widetilde{U}_T(\vec{\mathsf{p}})$ in its $\alpha$-cut representation, by applying it to the endpoints of the $\alpha$-cuts of the given fuzzy probabilities. If we represent fuzzy numbers using $N$ $\alpha$-cuts, we will hence need to run our chosen unreliability algorithm $2N$ times to compute the fuzzy unreliability, introducing a merely constant overhead. Therefore, we may conclude that any unreliability algorithm is efficient for fuzzy fault trees whenever it is efficient for ordinary ones. Using the state-of-the-art (modularised) BDD-based unreliability algorithm \cite{basgoze2022BDDs}, we confirm this conclusion with the experiments presented in the subsequent section.

\section{Experiments}\label{sec:experiment}

\begin{figure}[t]
\begin{floatrow}
\ffigbox{%
\includegraphics[width=\linewidth]{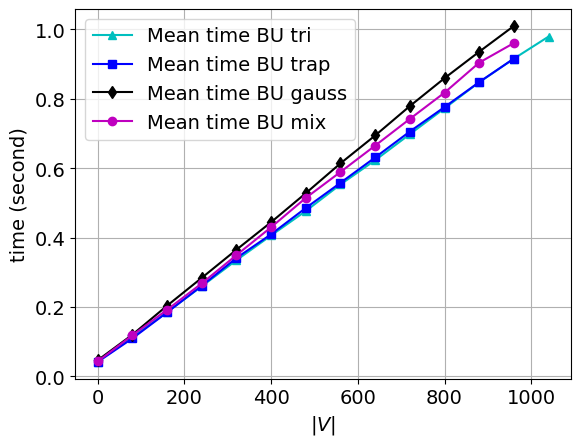}
}{%
 \caption{Mean time (in seconds) of performing bottom-up algorithm. Groups of generated trees are determined by the number of nodes $|V|$.}
    \label{fig:mean_time}
}
\capbtabbox{
  \begin{tabular}{lc} 
  Source   &  $|V|$  \\ \hline
\cite{stamatelatos2002fault} Fig. 12-7     & 10 \\
\cite{stamatelatos2002fault} Fig. 10-14    & 11 \\
\cite{MENTES2011application} Alt A1     & 17 \\
\cite{MENTES2011application} Alt A2  & 22  \\
\cite{MENTES2011application} Alt A6      & 31  \\
\cite{yazdi2017failure} Fig. 3 & 35 \\
\cite{MENTES2011application} Alt A12    & 39 \\
\cite{Sadiq2008predicting} Fig. 2  & 42 \\
\cite{MENTES2011application} Alt A11  & 45  \\
\cite{SHU2006using} Fig. 5    & 50  \\
  \end{tabular}
}{%
  \caption{FTs from the benchmarks used as subtrees for generation.}%
  \label{table:FTs}
}
\end{floatrow}
\end{figure}

We evaluate the performance of our two methods for fuzzy unreliability: 
the linear-time bottom-up algorithm for the special case of tree-structured FTs, and the general approach for DAG-structured FTs based on \Cref{theo:DAG-algo}. 

\subsection{Performance evaluation for tree-shaped FTs}

By the linear-time complexity of the bottom-up method,
computing the unreliability for tree-structured FTs should be feasible even for very large sizes.
To confirm this hypothesis, we therefore need to consider particularly large fault trees.
However, obtaining large FTs for real-world systems is infeasible due to confidentiality reasons, and the fact that we require tree-structured and non-dynamic FTs for this experiment. 
For this reason, we generate a custom benchmark of large tree-structured FTs, using the set of ten FTs from the literature shown in Table~\ref{table:FTs}. 
This method is inspired by the one used in \cite{lopuhaa2023attack}.

\myparagraph{Generation of large tree-structured FTs.} 
We generate large FTs using two ways of combining two FTs $T_1, T_2$ as shown in Fig.~\ref{fig:combineFTs}, starting from the ten FTs shown in \Cref{table:FTs}: 
\begin{enumerate}
    \item \emph{Horizontal combination}. We introduce a new root node with a random gate and add two edges: one edge from the new root to $R_{T_1}$ and another one from  the new root to $R_{T_2}$. The resulting system is a larger system consisting of two subsystems side-by-side.
    \item \emph{Vertical combination}. We randomly pick a basic event $v$ from $T_1$ and replace $v$ with $T_2$. The resulting system is a larger system in which the subsystem $T_2$ becomes part of $T_1$. 
\end{enumerate}

\begin{wrapfigure}[13]{r}{0.35\textwidth}
\centering
\vspace{-5pt}
\includegraphics[width=3cm]{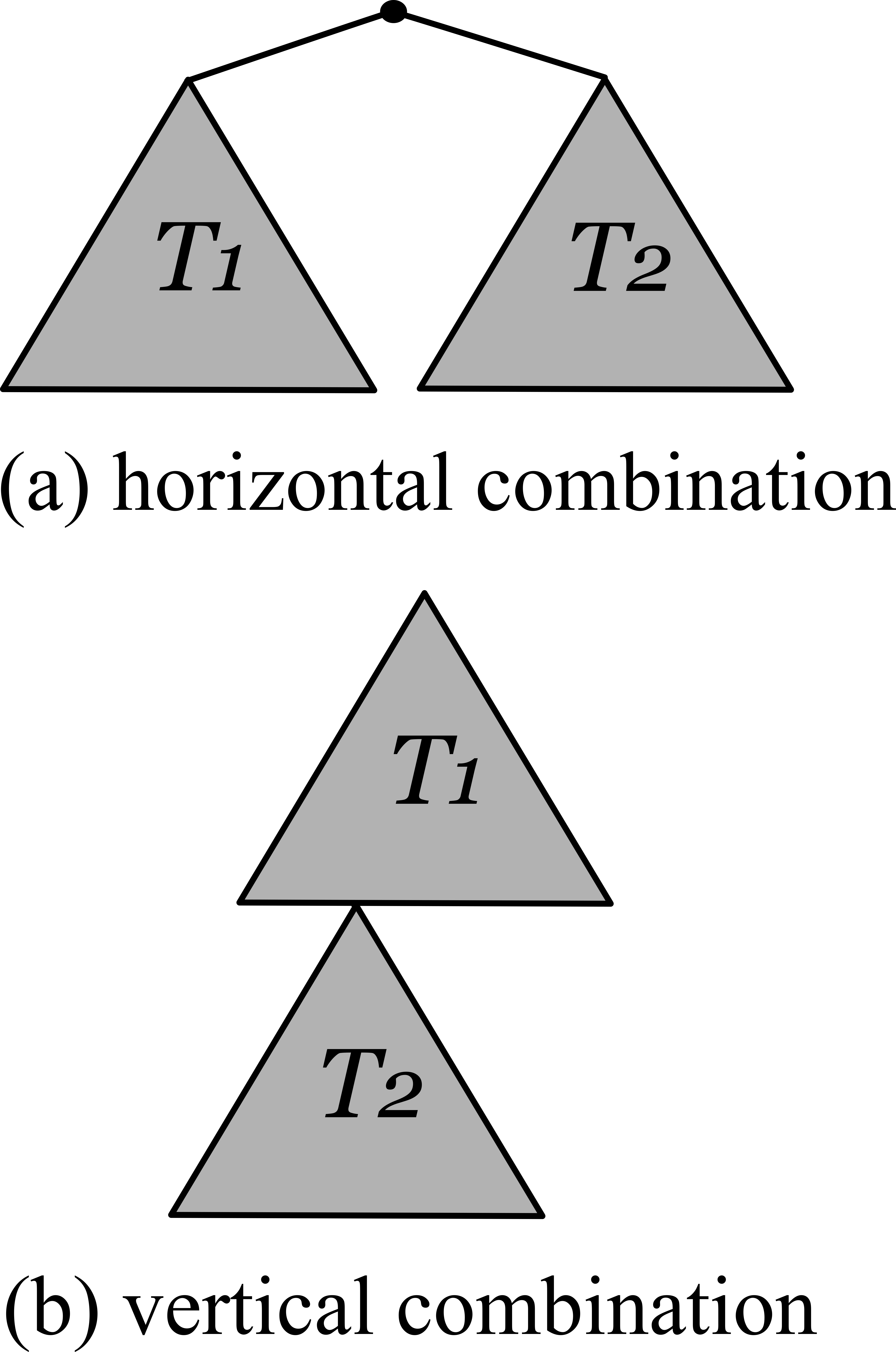}
  \caption{Combining FTs.}%
  \label{fig:combineFTs}
\end{wrapfigure}

To obtain a benchmark of \emph{fuzzy} FTs, we fuzzify the probabilities of all basic events using triangular, trapezoidal, truncated Gaussian fuzzy numbers, as well as a mixture of these, all centred at the original probabilities. 

The experiment was then conducted as follows. For each number $k \in$ $\{1,$ $\dots, N\}$, we randomly create large FTs using the two aforementioned combinations such that their number of nodes $|V|$ is at least $k$. These large FTs are then divided into groups by their size. Herein, we take $N=1000$, $P=80$ so the trees belong to group $\lceil|V_k|/P\rceil$. We run the bottom-up algorithm for every FT in each group and then derive the mean time for every group of large FTs. The mean time for running the bottom-up method for different BE fuzzy types is displayed in Fig.~\ref{fig:mean_time}. As expected from Section~\ref{sec:BU_alpha_cuts}, computation time is linear in FT size. Furthermore, computation is very fast, with even very large FTs taking only approximately $1s$ to calculate. Therefore, the proposed bottom-up algorithm is not only accurate, but also very efficient.


\subsection{Performance evaluation for general DAG-structured FTs}\label{sec:dag-structured-experiments}

\begin{wrapfigure}[14]{l}{0.45\textwidth} 
\vspace{-15pt} 
\centering
\includegraphics[width=1.0\linewidth]{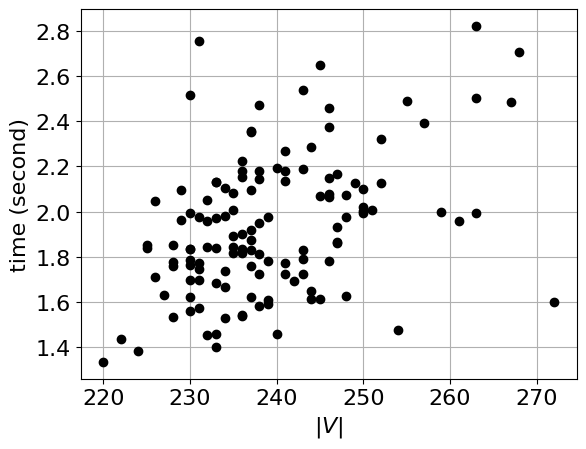} 
\caption{Performance of our method for general DAG-structured FTs.}
\label{fig:experiment_BDD}
\vspace{-5pt} 
\end{wrapfigure}

In this experiment, we employ Storm \cite{Hensel2022storm}, a state-of-the-art model checker that calculates FT unreliability via a BDD-based approach with modularization~\cite{basgoze2022BDDs}. 
\Cref{theo:DAG-algo} allows us to conveniently make use of this off-the-shelf tool to also compute the \emph{fuzzy} unreliability; see \Cref{sec:bdd-based-algo}.
Since we do not need particularly large FTs in this case, we evaluate our method on an established benchmark of 125 randomly generated fault trees from \cite{basgoze2022artifact}. The FTs in the benchmark have 239 nodes on average.
The DAG-structured FTs in this set have \emph{crisp} probabilistic status vectors. 
In order to obtain a benchmark of \emph{fuzzy} fault trees, we equip each basic event $b$ with the triangular fuzzy number $\triangle_{0.8p_b,p_b,1.2p_b}$, given the crisp  probability $p_b$. 
The number of $\alpha$-cuts is set to $10$. 
Figure~\ref{fig:experiment_BDD} shows the performance of our approach, as measured by the runtime of each instance. We observe that the elapsed time of computation for each tree is generally fast (less than 3 seconds). 

\newpage
\section{Related work}
\label{sec:related_works}

In the 1980s, fuzzy fault tree analysis was pioneered by Tanaka et al.~\cite{tanaka1983fault} and has since been extensively studied. Reviews are reported in~\cite{mahmood2013FFTA,ruijters2015FTA,kabir2017anoverview}. 

\subsubsection{Efficient fuzzy arithmetic.} Various studies 
offer solutions to the problem of efficiently computing the fuzzy arithmetic operations. 
When dealing with the multiplication of triangular or trapezoidal fuzzy numbers, it is common to use approximations that yield a fuzzy number of the same type as the original operands \cite{tanaka1983fault,peng2008approach,Liang1993FFTA}. 
However, these approaches yield approximate solutions that  generally overestimate the fuzzy probability of the top event. 
While over-approximations are conservative, 
there are no formal guarantees on their accuracy, rendering this method less informative in general. 
Alternatively, \cite{DSW1985fuzzy} provides an efficient way to calculate fuzzy arithmetic operations using the $\alpha$-cut representation of fuzzy numbers.

\subsubsection{Uncertainty quantification in FTs.}
Alternative frameworks for uncertainty quantification in fault trees include \emph{sensitivity analysis} \cite{ruijters2015FTA,rushdi1985uncertainty}, 
\emph{imprecise probability} (or ``probability intervals'') \cite{jacob2011uncertainty,jacob2012imprecise}, 
as well as Bayesian approaches \cite{prabhu2020uncertainty}, in which uncertainty in parameters is treated by viewing these parameters as random variables themselves. 
The approach using imprecise probability is subsumed by fuzzy fault tree analysis using \emph{interval fuzzy numbers}; see \Cref{sec:classes-of-fuzzy-numbers}. Sensitivity analysis considers how sensitively a given risk metric responds to varying model parameters. In contrast to fuzzy fault tree analysis, it is not concerned with a precise quantification of uncertainty in these parameters themselves. Finally, the Bayesian method of \cite{prabhu2020uncertainty} relies on Monte-Carlo simulation, which is an important further difference to fuzzy fault tree analysis, in addition to its differing conceptual foundation.

\subsubsection{Defuzzification approaches.}
A widely studied approach \cite{lin1997hybridFTA,KUMAR2022ffta} to fuzzy fault tree analysis
is to first obtain fuzzy probabilities for all basic events from expert opinions, then ``defuzzify'' these to a crisp values, and finally compute the probability of the top-level event as usual.  
Here, fuzzy theory is only used as an intermediate step for converting and aggregating expert knowledge to precise numerical probability values, and does not provide any uncertainty quantification at the system level.

\subsubsection{Linguistic variables.} 
The concept of \emph{linguistic variables} is used in fuzzy fault tree analysis to obtain fuzzy probabilities of basic events from expert opinions \cite{bowles1995application,lin1997hybridFTA,pan2007assessing,yuhua2005estimation}. A linguistic variable is a variable whose values range over a finite set of natural language descriptions. Each such description is then assigned a fuzzy number, modelling its inherent ambiguities. In the case of fuzzy fault trees, the linguistic variables of interest concern the probabilities of basic events, ranging from ``very low'' over ``medium'' to ``very high''.

\subsubsection{Case studies.}
An overview of case studies using uncertainty quantification in fault trees, and fuzzy fault trees in particular, is given in \cite{yazdi2019uncertainty}.



\section{Conclusion and future work}

Based on our rigorous definition of fuzzy fault trees (\Cref{def:fuzzy_unreliability}), \Cref{theo:DAG-algo} enables extending any fault tree unreliability algorithm to this setting---in a simple, correct and efficient manner. 
In addition, the efficacy of this method is confirmed by the experiments presented in \Cref{sec:experiment}.

There remain several open problems for future work. 
Most importantly, our rigorous formulation of fuzzy fault tree analysis may serve as a basis for a thorough and critical comparison to other frameworks for uncertainty quantification---which framework is most appropriate in which context?
Moreover, finding efficient algorithms for purely probabilistic (``Bayesian'') uncertainty quantification in fault trees remains an open direction. 
Finally, an interesting avenue for future research is to extend our results for fuzzy fault tree analysis to related risk models, such as dynamic fault trees (DFTs), as well as general DAG-structured attack trees. 




\myparagraph{Acknowledgements.}
This work was partially funded by the NWO grants NWA.1160.18.238 (PrimaVera), and  KICH1.ST02.21.003 
  (ZORRO), 
the European Union's Horizon 2020 research and innovation programme under the Marie Sk\l{}odowska-Curie grant agreement No 101008233, the ERC Proof-of-Concept grant 101187945 (\emph{RUBICON}), and the ERC Consolidator Grant 864075 (\emph{CAESAR}).%

\myparagraph{Data Availability.} Scripts to reproduce our experimental evaluation are archived and publicly available at \cite{dang_2025_15337097}.

%
%
%
\bibliographystyle{splncs04}
\bibliography{references}

\appendix

\section{Appendix}\label{sec:appendix}

\subsection{Proof of \Cref{lemma:zadeh-extension-via-alpha-cuts}}\label{sec:proof-of-lemma-on-zadeh-extensions}
To simplify notation, let 
$$ g: \mathbb{R} \to [0,1], \quad g(x) := \min\limits_{i=1,\dots,n}\ \mathsf{x}_i[x_i].  $$
By \Cref{def:extension_principle}, 
\begin{align*}
    \widetilde{f}(\mathsf{x}_1, \dots, \mathsf{x}_1)^{(\alpha)}
    &= \{y\in \mathbb{R} \mid \widetilde{f}(\mathsf{x}_1, \dots, \mathsf{x}_1)[y] \geq \alpha \} \\
    &= \left\{y\in \mathbb{R} \left\vert  \underset{\substack{x \in I^n,\;
    f(x)=y}}{\sup}g(x)  \geq \alpha \right. \right\}
\end{align*}

Since by \Cref{def:fuzzy-number}, each $\mathsf{x}_i$ is upper-semicontinuous and compactly supported as a function $\mathbb{R} \to [0,1]$, the same holds for $g$. 
Moreover, the set over which the supremum is taken is closed, as $f$ is assumed to be continuous and $I$ is closed.
Now, any compactly supported upper-semicontinuous function attains its maximum on a closed set, and therefore,
\begin{align*}
    \widetilde{f}(\mathsf{x}_1, \dots, \mathsf{x}_1)^{(\alpha)}
    &= \left\{y\in \mathbb{R} \:\left\vert\:  \underset{\substack{x \in I^n,\;
    f(x)=y}}{\max}g(x)  \geq \alpha \right. \right\} \\
    &= \left\{y\in \mathbb{R} \:\left\vert\:  \exists x \in I^n: \,f(x)=y, g(x)\geq \alpha \right. \right\} \\
    &= \left\{y\in \mathbb{R} \:\left\vert\:  \exists x \in I^n\, \forall i\in\{1,..., n\}:\, \mathsf{x}_i[x_i]\geq \alpha, f(x)=y \right. \right\} \\
    &= \left\{y\in \mathbb{R} \:\left\vert\:  \exists x \in \prod_{i=1}^n \mathsf{x}_i^{(\alpha)},\, f(x)=y \right. \right\} \\
    &= \left\{f(x) \:\left\vert\: x \in \prod_{i=1}^n \mathsf{x}_i^{(\alpha)} \right. \right\},
\end{align*}
for all $\alpha\in [0,1]$
This shows the first part of the claim. For the second part, note that since $f$ is continuous, the above image of the product of the $\alpha$-cuts under $f$ must be connected and hence be an interval. Finally, the endpoints of these intervals are given by the value of $f$ at the endpoints of the $\alpha$-cuts, by the monotonicity of $f$ in each argument.

\subsection{Proof of \Cref{theo:DAG-algo}}\label{sec:proof-DAG-algo}

Once we know that the unreliability function $U_T$ is continuous and monotonically increasing, the claim follows directly from \Cref{lemma:zadeh-extension-via-alpha-cuts}. The continuity of $U_T$ is clear, since by \Cref{eq:UT}, $U_T$ is a polynomial function.

The monotonicity of $U_T$, on the other hand, follows from \Cref{lemma:monotonicity} below, since, as a composite of AND and OR gates, the structure function $S_T(R_T, -)$ of $T$ is a \emph{monotone} Boolean function. (Recall that a map $f:X\to Y$ between partially ordered sets $X,Y$ is called \emph{monotone} 
if for all $x_1, x_2\in X$, $x_1\leq x_2$ implies $f(x_1)\leq f(x_2)$.)

In the lemma below, $\mathrm{Bern}(p)$ denotes the Bernoulli distribution on the set $\mathbb{B}:=\{0, 1\}$, 
which assigns probability $p\in [0,1]$ to $1$ and $1-p$ to $0$. 
Moreover, if $(X, \leq)$ is a partially ordered set (such as $\mathbb{B}=\{0, 1\}$ or $[0,1]$) 
and $n\in\mathbb{N}$, 
we view $X^n$ as a partially ordered set in which $(x_1, \dots, x_n) \leq (y_1, \dots, y_n)$ 
if and only if $x_i\leq y_i$ for all $i\in \{1, \dots n\}$. (In particular, $U_T: [0,1]^{BE_T}\to [0,1]$ is monotone in this partial order if and only if it is monotonically increasing in each variable.)
We may now state, and prove:

\begin{lemma}\label{lemma:monotonicity}
    Let $n\in \mathbb{N}$ and let $f: \mathbb{B}^n \to \mathbb{B}$ be monotone Boolean function.
    Then 
        $$ \rho : [0,1]^n \to [0,1], \;\; p \mapsto \mathbb{P}_{A_1,\dots, A_n \,\sim\, \bigotimes_{i=1}^n \mathrm{Bern}(p_i)}[f(A_1, \dots, A_2)] $$
    is monotone, too.
\end{lemma}
\begin{proof}
    Let $Q:= \bigotimes_{i=1}^{n} \mathrm{Unif}[0,1]$ 
    be the joint distribution of $n$ independent uniformly distributed
    random variables on the unit interval. 
    Let $x_i: [0,1]^n \to [0,1]$ be the $i$-th projection.
    Then for all $p\in [0,1]^n$,
    \begin{align*}
        \rho(p) &= \mathbb{P}_{A_1,\dots, A_n \,\sim\, \bigotimes_{i=1}^n \mathrm{Bern}(p_i)}[f(A_1, \dots, A_2)]\\
                &= \mathbb{P}_{x_1,\dots, x_n \,\sim\, Q}[f(\{x_1 \leq p_1\}, \dots, \{x_n \leq p_n\})].
    \end{align*}
    Hence, $\rho$ is the composite, 
        $$ [0,1]^n \xrightarrow{(\{x_i \leq \cdot\})_{i=1}^n} \mathcal{B}([0,1]^n)^n \xrightarrow{f[\,\cdot\,]} \mathcal{B}([0,1]^n) \xrightarrow{Q} [0,1],$$
    of monotone maps. 
    (Here, $\mathcal{B}([0,1]^n)$ is the Borel $\sigma$-algebra on $[0,1]^n$.) 
    Therefore, $\rho$ is itself monotone. 
\end{proof}

\end{document}